\documentclass[11pt,letterpaper]{amsart}

\usepackage[english]{babel}
\usepackage[utf8x]{inputenc}
\usepackage[T1]{fontenc}
\usepackage[normalem]{ulem}

\usepackage{amsmath}
\usepackage{amsfonts}
\usepackage{amsthm}
\usepackage{graphicx}
\usepackage{enumerate}
\usepackage{physics}
\usepackage{marginnote}
\usepackage[colorinlistoftodos]{todonotes}
\usepackage[colorlinks=true, allcolors=blue]{hyperref}

\numberwithin{equation}{section}
\newtheorem{theorem}{Theorem}[section]
\newtheorem{corollary}[theorem]{Corollary}
\newtheorem{lemma}[theorem]{Lemma}

\newtheorem{proposition}[theorem]{Proposition}

\newtheorem{problem}[theorem]{Problem}

\theoremstyle{definition}

\theoremstyle{remark}
\newtheorem{remark}[theorem]{Remark}
\newtheorem{example}[theorem]{Example}

\newcommand{\C}{\mathbb{C}}
\newcommand{\N}{\mathbb{N}}
\newcommand{\bn}{\mathbb{N}}
\newcommand{\R}{\mathbb{R}}
\newcommand{\br}{\mathbb{R}}
\newcommand{\Z}{\mathbb{Z}}
\newcommand{\bz}{\mathbb{Z}}

\newcommand{\Q}{\mathbb{Q}}
\newcommand{\bq}{\mathbb{Q}}

\newcommand{\ch}{\mathbf 1}

\newcommand{\supp}{\operatorname{supp}}

\title{Are MSF Wavelets Minimally Supported?}

\author{Marcin Bownik}
\address{Department of Mathematics,
University of Oregon,
Eugene, OR 97403--1222, USA}
\email{mbownik@uoregon.edu}

\author{Ziemowit Rzeszotnik}
\address{Mathematical Institute, University of Wroc\l aw, 50--384 Wroc\l aw, Poland}
\email{zioma@math.uni.wroc.pl}

\author{Darrin Speegle}
\address{Department of Mathematics and Statistics, Saint Louis University, St. Louis, MO 63103, USA}
\email{speegled@slu.edu}

\thanks{The first author was partially supported by NSF grant DMS-2349756.}

\begin{document}

\begin{abstract}
Larson's problem \cite[Problem 3]{Lar} asks ``Must the support of the Fourier transform of a wavelet contain a wavelet set?". We give an affirmative answer to a non-measurable variant of this question by proving that the Fourier transform of a wavelet must contain a possibly non-measurable wavelet set. We also provide background results on Larson's problem and propose two new related problems.
\end{abstract}
\maketitle

\section{Introduction}

In the early 1990's, groups led by David Larson and Guido Weiss were studying the mathematical underpinnings of wavelets. As would be expected, the two groups came up with different terms for objects that were of fundamental interest in their investigations. One such example is a wavelet with the property that 
%the modulus of 
its Fourier transform (i.e. frequency) is the indicator function of a set. Larson focused on the frequency support of the wavelets, and called the frequency support of such wavelets \textit {wavelet sets \cite{DaiLar98, DaiLarSpe97}.} Weiss focused on wavelets whose modulus of the frequency is the indicator function of a set, and called such wavelets {\textit{mimimally supported  frequency} (MSF) wavelets \cite{HerWei96, WeiWil01}.} The reasoning behind Weiss' nomenclature is that minimally supported wavelets have support with the smallest possible Lebesgue measure.\footnote{See the next section for proofs of facts in the introduction, as well as more definitions.} In the middle of 1990's, Larson asked whether minimally supported wavelets are also minimally supported with respect to set inclusion \cite{Lar}. That is, if $\psi$ is a wavelet whose Fourier transform has support $E$, and $\psi$ is not an MSF wavelet, is there a set $F \subset E$ with the Lebesgue measure of $E\setminus F$ positive and such that $F$ is the support of the Fourier transform of a wavelet? The most natural way to prove that MSF wavelets are minimally supported with respect to set inclusion would be to show that the support of the Fourier transform of a wavelet must contain a wavelet set.
This was formalized as an open problem in \cite{Lar}.

\begin{problem}\label{mainprob}
Must the support of the Fourier transform of a wavelet contain a wavelet set?
\end{problem}

The motivation for this problem (in addition to the nomenclature presented above) is two-fold. First, Larson had a program for using operator theory to build more complicated wavelets from simpler wavelets. In particular, Larson, together with Dai, had an operator theoretic approach to constructing the Lemarie-Meyer wavelets from wavelet sets \cite{DaiLar98}. Their construction is similar in spirit to the construction presented in \cite{HerWei96}. Larson and Dai also proved that  operator theoretic techniques could be used on more complicated families of wavelet sets to build more general wavelets. A natural question is whether {\it{every}} wavelet can be built using these techniques. However, since the techniques always produce a wavelet whose frequency support is contained in the union of the underlying wavelet sets and contains a wavelet set, a necessary condition to be able to construct every wavelet this way would be a positive solution to Problem \ref{mainprob}.

The second motivation came later, and is also more tenuous. It is easy to see that if $\pi$ is a rotation on the plane, then there exists an $S = S_\pi\subset \R^2$ such that $S$ tiles $\R^2$ via translations along $\pi (\Z^2).$ The Steinhaus tiling problem asks whether there exists a set $S \subset \R^2$ such that for every rotation $\pi$, $S$ tiles $\R^2$ when translated by the elements of $\pi (\Z^2)$. The Steinhaus problem was solved in \cite{JacMau02, JM}, where a non-measurable set $S$ was constructed. It remains an open problem whether a measurable Steinhaus set in $\R^2$ can be constructed. For a higher dimensional variant of Steinhaus problem, see \cite{GoldMau}.

Similarly, we will see below in Proposition \ref{transanddiltileprop} that if $E$ is the support of a wavelet, $E$ must contain a subset which tiles by translations along $\Z$ and a different subset which tiles by dilations of powers of 2. Problem \ref{mainprob} asks whether there is a single set which does these simultaneously. In the Steinhaus tiling problem and Problem \ref{mainprob}, we are asking to find a single set which tiles via multiple actions. The differences are many: the Steinhaus problem has uncountably many tilings, while we only have two, the Steinhaus problem only considers translational tilings while our tilings are more diverse, and our problem considers the relationship between functional tilings and set tilings. 

It is worth mentioning two other similar works on simultaneous tilings. Han and Wang \cite{HanWang} showed that simultaneous translational tiling by two lattices of the same volume always exists. Kolountzakis and Papageorgiou \cite{KolPa} studied simultaneous functional tilings with respect to a collection of lattices.

The main result in this paper shows that a possibly non-measurable wavelet set exists in the support of the Fourier transform of any wavelet. The proof relies on a result of Isbell \cite{Isb55, Isb62} on the existence of a positive diagonal of a doubly stochastic matrix. The original Larson's problem, whether such a set can be chosen to be measurable, remains open.

\section{Notation and Background}

A {\it (dyadic, orthonormal) wavelet} is a function $\psi \in L^2(\R)$ such that 
\[
\{2^{j/2} \psi(2^jx + k): j, k \in \Z\}
\]
is an orthonormal basis for $L^2(\R)$. In this paper, we will be concerned with the Fourier transforms of wavelets. The Fourier transform we use is given for functions $f\in L^1(\R) \cap L^2(\R)$ by
\[
\hat f(\xi) = \int_{\R} f(x) e^{-2\pi i \xi x}\, dx.
%\qquad {\text {(I think)}}
\]

If $\psi$ is a wavelet, then the following four equations are satisfied \cite{HerWei96}:
\begin{align}
    \sum_{k \in \Z} \abs{\hat \psi(\xi + k)}^2 &= 1 \qquad\text{for a.e. }\xi\in\R, \label{eq:1} \\
    \sum_{j \in \Z} \abs{\hat \psi (2^j\xi)}^2 &= 1 \qquad\text{for a.e. }\xi\in\R, \label{eq:2}\\
    t_q(\xi) := \sum_{j = 0}^\infty \hat\psi(2^j\xi) \overline{\hat\psi(2^j(\xi + q)} &= 0 \qquad\text{for a.e. }\xi\in\R\text{ for any }q\in 2\Z+1, \label{eq:3}  \\
     \sum_{k \in \Z} \hat \psi(2^j(\xi + k)) \overline{\hat \psi (\xi + k)} &= 0 \qquad\text{for a.e. }\xi\in\R\text{ and for any } j \ge 1. \label{eq:4}
\end{align}
Moreover, if equations \eqref{eq:1} and \eqref{eq:2} hold, then either one of conditions \eqref{eq:3} and \eqref{eq:4}  implies that $\psi$ is a wavelet, see \cite{Bow01, HerWei96}.

We note here the similarity between \eqref{eq:1}, \eqref{eq:2}, and geometric notions of tilings. In the special case that $\hat \psi = \ch_E$, \eqref{eq:1} holds exactly when $\{E + k: k\in \Z\}$ is a measurable tiling of $\R$ in the sense that 
\begin{equation}
\label{tt}
\begin{aligned}
    m\bigg( \R \setminus \bigcup_{k \in \Z} \left(E + k\right) \bigg) &= 0 \qquad\text{and}\\
    m\left((E + k) \cap (E + k^\prime)\right) &= 0 \qquad\text{for all } k\not= k^\prime \in \Z,
\end{aligned}
\end{equation}
where $m$ is the Lebesgue measure on $\R$. Similarly, \eqref{eq:2} holds exactly when $\{2^j E: j\in \Z\}$ is a measurable tiling of $\R$,
\begin{equation}
\label{dt}
\begin{aligned}
    m\bigg( \R \setminus \bigcup_{j \in \Z} 2^jE \bigg)& = 0 \qquad\text{and}\\
    m\left((2^jE) \cap (2^{j'}E)\right) &= 0 \qquad\text{for all } j\not= j^\prime \in \Z,
\end{aligned}
\end{equation}
A set $E \subset \R$, which satisfies both \eqref{tt} and \eqref{dt}, is called {\it a wavelet set}.
It can be shown that \eqref{eq:3} and \eqref{eq:4} are superfluous whenever $\hat \psi = \ch_E$, and we have the following theorem, see  \cite[Corollary 2.4 in Chapter 7]{HerWei96} for details.

\begin{theorem}\label{msfchar}
    Let $\psi$ be a function in $L^2(\R)$ such that $|\hat \psi| = \ch_E$. The function $\psi$ is a wavelet if and only if $E$ is a wavelet set. 
    
    In particular, if $W$ is a wavelet set and $m: W \to \C$ is a unimodular function, then $\psi$ given by $\hat\psi(\xi)=m(\xi) \ch_W(\xi)$ is an MSF wavelet. 
\end{theorem}

A wavelet $\psi$ as in Theorem \ref{msfchar} is called {\textit{mimimally supported  frequency}} (MSF). An MSF wavelet $\psi$ has the smallest possible Lebesgue measure of the frequency support $\supp \hat\psi$, which is equal to $1$. 
Here and throughout the paper, the \textit{support} of a measurable function $f:\R \to \C$ is defined up to a set of measure zero as 
\[
\supp f =\{x \in \R: f(x) \not= 0\}.
\]
The support of an MSF wavelet is a wavelet set. Consequently, Theorem \ref{msfchar} provides a correspondence between wavelet sets and MSF wavelets.

\begin{example}
    The following are three basic examples of wavelet sets. 
    \begin{enumerate}
        \item (Shannon wavelet set) $[-1, -1/2] \cup [1/2,1]$,
        \item $[-2/3, -1/3] \cup [2/3, 4/3]$, and
        \item (Journ\'e wavelet set) $[-16/7, -2] \cup [-1/2, -2/7] \cup [2/7, 1/2] \cup [2, 16/7]$.
    \end{enumerate}
\end{example}

Returning to interpreting equations \eqref{eq:1} and \eqref{eq:2} in terms of tiling, a function $f:\R \to \C$ is said to tile $\R$ by translations if 
\[ \sum _{k \in \Z} f(\xi + k) = 1 \qquad\text{for a.e. }\xi\in\R,
\]
with the convergence being absolute for almost every $\xi$. Similarly, we say that $f$ tiles $\R$ by dilations if
\[ \sum _{j \in \Z} f(2^j\xi) = 1 \qquad\text{for a.e. }\xi\in\R.
\]
Equations \eqref{eq:1} and \eqref{eq:2} state that $\abs{\hat \psi}^2$ tiles $\R$ by translations and by dilations. In light of Problem \ref{mainprob}, it is natural to ask the following question, a positive answer to which would imply a positive answer to Problem \ref{mainprob}. 

\begin{problem}\label{stronglarson}
    If $f$ is a measurable function which tiles by translations and dilations, is there a measurable set $W$ contained in the support of $f$ which tiles by translations and dilations? 
\end{problem}

It may not be immediately clear how to leverage equations \eqref{eq:3} and \eqref{eq:4} to make Problem \ref{mainprob} easier. However, note the following consequences of \eqref{eq:3} and \eqref{eq:4}, which provide additional geometric information about the support of the Fourier transform of a wavelet. 

\begin{proposition}\label{prop:geomsupport}
  Let $\psi$ be a wavelet whose Fourier transform is supported on $E$.
    \begin{enumerate}
        \item If $\xi \in E$ and $\xi + k \in E$ for an integer $k \not= 0$, then there exists $0\not= j \in \Z$ such that $2^j\xi \in E$ and $2^j(\xi + k) \in E$, where $2^j k \in \Z$.
        \item If $\xi \in E$ and $2^j \xi \in E$ for an integer $j\not= 0$, then there exists $0\not=k\in \Z$ such that $\xi + k \in E$ and $2^j(\xi + k) \in E$, where $2^jk \in \Z$.
    \end{enumerate}
\end{proposition}

\begin{proof}
    For (1), write $k = 2^J q$, where $J \ge 0$ and $q$ is odd. Applying \eqref{eq:3} to the point $\xi^\prime = 2^{-J} \xi$, we see that 
    \[
    \hat\psi(2^J \xi^\prime) \overline{\hat\psi(2^J(\xi^\prime + q))} = \hat\psi(\xi) \overline{\hat\psi(\xi + k)}
    \]
    is one of the terms in the sum of $t_{q}$. Since $t_{q} = 0$ almost everywhere, for almost all $\xi \in E \cap (E - k)$, there is a $0\le j \not= J$ such that 
    \[
    \hat\psi(2^j \xi^\prime)  \overline{\hat\psi(2^j(\xi^\prime + q))}
    =
     \hat\psi(2^{j-J} \xi)  \overline{\hat\psi(2^{j-J}(\xi + k))}
    \not= 0.
    \]
    By definition of $q$, we also can conclude that $2^{j-J} k \in \Z$. This proves (1). 
    
    Part (2) follows similarly by using \eqref{eq:4} instead. Indeed, suppose first that $j\ge 1$. Since we have a non-zero term in the sum \eqref{eq:4} corresponding to $k=0$, there exists another non-zero term corresponding to some $k\ne 0$. Likewise, if $j<0$, then applying \eqref{eq:4} with $\xi$ replaced by $\xi'=2^j\xi$ and $j$ replaced by $j'=-j$ yields the required conclusion.
\end{proof}

In order to construct an example that illustrates the additional geometric condition that \eqref{eq:3} imposes, we will need some tools. To simplify statements, we will say that a set $E$ \textit{packs} by translations if $(E + k) \cap (E + k^\prime)$ has measure zero whenever $k \not= k^\prime$ are integers. 

\begin{theorem}\cite{IP98}\label{ionascupearcy}
  Let $F$ be a measurable subset of $\R$. Then, $F$ is a subset of a wavelet set if and only if
  \begin{enumerate}
      \item $F$ packs by translations,
      \item $F$ packs by dilations,
      \item there exists a measurable set $U \supset F$ such that $U$ packs by translations and $U$ tiles by dilations, and 
      \item there exists a measurable set $V \supset F$ such that $V$ packs by dilations and $V$ tiles by translations. 
  \end{enumerate}
  Moreover, if the four conditions are met, then the wavelet set can be chosen to be a subset of $U \cup V$.
\end{theorem}

Note that Theorem \ref{ionascupearcy} implies the following.
\begin{remark}
   Larson's problem has a positive solution if for every wavelet $\psi$, there are measurable sets $U, V \subset \supp\hat \psi$ such that $U$ packs by translations and tiles by dilations, and $V$ packs by dilations and tiles by translations.
\end{remark}
    
Theorem \ref{ionascupearcy} also generalizes to translations along other lattices and dilations by invertible $n \times n$ matrices. It has been used to construct wavelet sets of various types, see for example \cite{BowSpe25, BowSpe02}. The following corollary has been implicitly used in constructions, but has never appeared as a stand-alone fact, so we provide a proof as well. 

\begin{corollary}\label{speeglesubset}
  Let $X$ be a finite collection of points such that
\begin{align*}
 X \cap \Z &= \emptyset,\\
 X \cap (X + k) &= \emptyset \,\, \text{whenever $0\not=k\in \Z$, and} \\
 X \cap 2^jX &= \emptyset \,\, \text{whenever $0\not=j\in \Z$,}
\end{align*}
then there exists $\delta > 0$ such that whenever $0 < \epsilon_x < \delta$ for all $x \in X$, the set 
\[
 \bigcup_{x \in X} (x - \epsilon_x, x + \epsilon_x)
\]
is contained in a wavelet set.
\end{corollary}

\begin{proof}
    From the conditions on $X$, we can choose $\delta > 0$ such that $F = \bigcup_{x \in X} B_\delta(x)$ packs by both translations and dilations. By reducing $\delta$ if necessary, we can also see that 
    \begin{equation}\label{empty}
            \overline{\bigcup_{k \in \Z} (F + k)} \cap \Z = \emptyset
    \end{equation}
    and 
    $\R \setminus \bigcup_{j \in \Z} 2^j F$ contains an interval.  

    Let 
    \[
    U_1 = \left([-2, -1] \cup [1,2]\right) \setminus \bigcup_{j \in \Z} 2^j F. 
    \]
    By \eqref{empty}, there exists a $J \in \Z$ such that $2^J U_1 \cap \overline{\bigcup_{k \in \Z} (F + k)} = \emptyset$. Therefore, $U = 2^J U_1 \cup F$ packs by translations and tiles by dilations.

    To see the existence of $V$, note that $V_1 = \R \setminus \bigcup_{j \in \Z} 2^j F$ satisfies $2 V_1 = V_1$. Therefore, $V_1$ contains an interval of length greater than 2, which packs by dilations. Let 
    \[
    V_2 = [0, 1] \setminus \bigcup_{k \in \Z} \left(F + k\right)
    \]
    and see that there exists a $k \in \Z$ such that $V_2 + k \subset V_1$. Therefore, $V = F \cup \left(V_2 + k\right)$ packs by dilations and tiles by translations. By Theorem \ref{ionascupearcy}, there exists a wavelet set $W$ which contains $F$ (and is contained in $U \cup V$). Note that if we replace $\delta$ by $0 < \epsilon_x \le \delta$, then $\cup_{x \in X} B_{\epsilon_x}(x) \subset F \subset W$, and the result is proven.
\end{proof}

Next, we will see a typical application of Corollary \ref{speeglesubset}, which we use to provide the promised example of the extra geometry that \eqref{eq:3} imposes on the support.

\begin{example}
    Let $X = \{1/5, 12/5, 34/5\}$. Then $X$ satisfies the conditions of Corollary \ref{speeglesubset}; let $\delta > 0$ be as in the conclusion. (The points in $X$ were chosen to satisfy $12/5 = 2(1/5) + 2$ and $34/5 = 2^2 (1/5) + 6$.)
 %Need to clean this up. Section 3 uses A, B, C for subsets of R with no real meaning. I like that better than F, G, and H that we are using in this section. We use A as a dilation matrix one time, but I don't think that will cause any issues 
    Let $G = (1/5 - \delta/4, 1/5 + \delta/4)$.
    Let
    \begin{align*}
        F &= G \cup 2(G + 1) \cup \left(4 G + 6 \right) \,\, \text{and} \\
        H &= G \cup(G + 1) \cup 4G \cup 2(G + 1) \cup (2(G + 1) + 1)\cup (4G + 6),
    \end{align*}
    reducing the size of $\delta$ if necessary to insure that the unions are disjoint.  
    
    It can be useful to imagine $H$ in matrix form, where sets in the same column tile the same subset of $\R$ via dilations, and sets in the same row tile the same subset of $\R$ via translations, as below:
    \begin{equation}\label{matrixrep}
    \begin{bmatrix}
    G&G + 1&\\
    &2(G + 1)&2(G + 1) + 1\\
    4G&&4G + 6
    \end{bmatrix}.
    \end{equation}

Let $W$ be a wavelet set containing $F$. 
Define $f$ via

    \[
    f(\xi) = \begin{cases}
        1&\xi \in W\setminus F,\\
        \frac 12&\xi \in H,\\
        0&\text{otherwise}.
    \end{cases}
    \]
    We see that $f$ tiles by translations and by dilations, and the support of $f$ is $H \cup W$. However, we note that for $\xi \in G \subset H \cup W$, 
    $\xi + 1 \in H \cup W$, but for no $j \in \Z \setminus \{0\}$ is both $2^j \xi$ and $2^j(\xi + 1)$ in $H \cup W$. Therefore, by Proposition \ref{prop:geomsupport} there is no wavelet whose Fourier transform has support equal to $H \cup W$.

    We can see, however, that the support of $f$ contains a wavelet set; indeed, the support of $f$ was constructed to contain $W$, which is a wavelet set.
\end{example}

Visualizing the support of $\hat \psi$ as a matrix as in \eqref{matrixrep} is a key to the results in this paper. It will be utilized extensively in the proof of Theorem \ref{dar}.
Two more results relating to Problem \ref{mainprob} are now presented.

First, let $\psi$ be a wavelet, and let $E$ be the support of $\hat \psi$. In \cite{RzeSpe02}, the second and the third authors showed that using \eqref{eq:3}, the organizational trick of \eqref{matrixrep}, and the additional assumptions that 
    \begin{align*}
        \sum_{k \in \Z} \ch_E(\xi + k) &\le 2
        \qquad\text{for a.e. }\xi\in\R,\\
        \sum_{j \in \Z} \ch_E(2^{j} \xi) &\le 2 
        \qquad\text{for a.e. }\xi\in\R,
    \end{align*}
the support $E$ contains a wavelet set. The idea was to show that the support of $\hat \psi$ partitions into sets of the form 
\[
G \cup \bigcup_{i} \left(F_i \cup (F_i + k_i) \cup 2^{j_i} F_i \cup 2^{j_i}(F_i + k_i)\right)
\]
with $j_i \ge 1$, $k_i\in\Z$, and $G = \{\xi : |\hat \psi(\xi) |= 1\}$. The existence of such partition implies that
\[
G \cup \bigcup_{i} F_i \cup 2^{j_i}(F_i + k_i)
\]
is a wavelet set.

Second, in \cite{BowRze}, the first and second authors showed that the support of the Fourier transform of an \textit{MRA} wavelet contains a wavelet set. Let $\psi$ be an MRA wavelet with associated scaling function $\phi$. 
 For background on these concepts, see for example \cite{HerWei96}. The authors showed specifically that the support $\hat \phi$ contains a scaling set whose associated wavelet set is contained in the support of $\hat \psi$. In particular, the authors showed that the support of the Fourier transform of an MRA wavelet contains a wavelet set \textit{associated with an MRA}. 

This leads to the following natural strong version of Problem \ref{mainprob}.

\begin{problem}\label{problemstrongdimension}
  Let $\psi$ be an orthonormal wavelet with the wavelet dimension function 
  \[
  D_\psi(\xi) = \sum_{j=1}^\infty \sum_{k\in \Z} |\hat\psi(2^j(\xi+k))|^2 \qquad \text{for a.e. }\xi\in\R. 
  \]
  Does the support of $\hat \psi$ contain a wavelet set whose associated MSF wavelet has the same dimension function as $\psi$? See \cite{BowRzeSpe01}, for example, for more information on wavelet dimension function.
\end{problem}

MRA wavelets are exactly those wavelets whose wavelet dimension function $D_\psi$ is identically 1. Hence, the result \cite[Theorem 2.8]{BowRze} shows that Problem \ref{problemstrongdimension} has a positive solution when $D_\psi$ is identically 1. Problem \ref{problemstrongdimension} is open for all other dimension functions.

We end this section with two simple observations. The first observation is that it is possible to \textit{separately} find two sets which tile $\R$ by translations and dilations, respectively.

\begin{proposition}\label{transanddiltileprop}
    Let $f$ be a non-negative, measurable function that tiles by translations. The support of $f$ contains a measurable set which tiles by translations. The same result is true with translations replaced by dilations.
\end{proposition}

\begin{proof}
    Let $E$ denote the support of $f$. Since $\ch_E \ge f$ \ a.e., we have 
    \[\sum_{k \in \Z} \ch_E(\xi + k) \ge 1 \qquad\text{a.e.}\]
    Thus, $\bigcup_{k \in \Z} (E + k) = \R$ up to a set of measure zero. 
    %\marginnote{We need an easy proof of this theorem} 
    
    Let $\{F_n\}_{n \in \N}$ be an enumeration of the sets $\{[0, 1] + k: k \in \Z\}$. Define $G_0 = \emptyset$. For each $n \ge 1$, define 
    \[
    G_{n + 1} = G_n \cup \left(\left(E \cap F_{n + 1}\right) \setminus \bigcup_{k \in \Z} (G_n + k) \right)
    \]
    and see that $G = \bigcup_{n = 1}^\infty G_n\subset E$ tiles by translations.

    The proof for dilations follows similarly by letting $\{F_n\}_{n \in \N}$ be an enumeration of $\{2^j W: j \in \Z\}$, where $W = [-1, -1/2] \cup [1/2, 1]$.
\end{proof}

\begin{proposition}\label{integralvalue}
    Let $\psi$ be a wavelet. Then
    \begin{align}
        \int_\R \left |\hat \psi(\xi)\right |^2\, d\xi &= 1, \label{int:1} \\
        \int_\R \left |\hat \psi(\xi)\right |^2\, \frac{d\xi}{|\xi|} &= \log 4. \label{int:2}
    \end{align}
\end{proposition}

\begin{proof}
    Use \eqref{eq:1} and \eqref{eq:2}.
\end{proof}

The second observation is that it is too much to hope that the support of a function that satisfies the conclusion of Proposition \ref{integralvalue} will always contain a wavelet set.

\begin{example}\label{ex:toostrong}
    Let $F =\left[-1/3, -1/6\right] \cup \left[1/3, 1/2\right] \cup \left[2, 8/3\right]$. It is easy to see that $F$ tiles by dilations, but $F$ does not tile by translations. Hence, neither $F$ nor any of its subsets is a wavelet set. That is, $F$ does not contain a wavelet set. Moreover, 
        \begin{align*}
        \int_\R \left |\ch_F(\xi)\right |^2\, d\xi &= 1,\\
        \int_\R \left |\ch_F (\xi)\right |^2\, \frac{d\xi}{|\xi|} &= \log 4.
    \end{align*}
 This implies that $|\hat \psi| \le 1$, \eqref{int:1}, and \eqref{int:2} do not together imply that the support of $\hat \psi$ contains a wavelet set. 
\end{example}

\section{Non-measurable Larson's problem}

In this section, we leverage facts about doubly-stochastic matrices to prove that if $f$ is a nonnegative function that tiles by translations and dilations, then its support contains a possibly non-measurable set $W$ that tiles by translations and dilations. 

Throughout this section, $J$ will denote an index set of the form $\{1, \ldots, N\}$ for some $N$ or $\N$. A square matrix $A = \left(a_{ij}\right)_{i, j \in J}$ with nonnegative entries is said to be \textit{doubly stochastic} if for every $k \in J$
\[
    \sum_{j \in J} a_{kj} =1 \qquad\text{and}\qquad
    \sum_{i \in J} a_{ik} =1.
\]
A diagonal of a square matrix
is a set of entries including exactly one from each row and one from each column. 
We will make essential use of the following fact due to Isbell \cite{Isb55, Isb62}. 

\begin{theorem}\label{birkhoff}
    Let $A = \left(a_{ij}\right)_{i, j \in J}$ be a nonnegative, doubly stochastic matrix. Then, there is a positive diagonal of $A$. That is, there is a permutation $\sigma$ of $J$ such that $a_{i\sigma(i)} > 0$ for all $i\in J$.
\end{theorem}

In the arguments provided below it is convenient to concentrate on irrational $\xi\in\br$. Thus, a set $C\subset\br$ described below is provided as a subset of $\br\setminus\bq$.
\begin{proposition}\label{propC}
 If $f:\br\to\br$ is non-negative and
\begin{equation} \label{p1}
\sum_{k\in\bz}f(\xi+k)= \sum_{j\in\bz}f(2^j\xi)=1 \qquad
\text{for a.e. }\xi\in\R,
\end{equation}
then there exists a set $C\subset\br\setminus \bq$ of full measure, such that \eqref{p1} holds for every 
$\xi\in C$ and $2^jC+q= C$ for all $j\in\bz$, $q\in\bq$.

\end{proposition}
\begin{proof}
By \eqref{p1} there is a set $F\subset\br$ of full measure such  that $\sum_{k\in\bz}f(\xi+k)=1$ for all $\xi\in F$
and a  set $G\subset\br$ of full measure such  that $\sum_{j\in\bz}f(2^j\xi)=1$ for all $\xi\in G$. Therefore, \eqref{p1}
holds for all $\xi\in F\cap G$. Moreover, $F\cap G$ has full measure, since $m((F\cap G)^c)=m(F^c\cup G^c)\le
m(F^c)+m(G^c)=0$.

Set $F'=F\cap(\br\setminus \bq)$ and $G'=G\cap(\br\setminus \bq)$ and 
consider
\[
C=\bigcap_{j\in\bz,q\in\bq}\Big(2^j(F'\cap G')+q\Big).
\]
Since $C$ is an intersection of a countable family of sets of full measure, an argument as above assures that $C$ is of full measure.

Clearly, if $\xi\in C$, then $\xi\in F'\cap G'$, so \eqref{p1} holds for such $\xi$.
Take any $j_0\in \Z$ and $q_0 \in \Q$. Then,
\[
2^{j_0}C+q_0=\bigcap_{j\in \Z, q\in \Q} (2^{j+j_0}(F' \cap G') + 2^{j_0}q+q_0) = C.
\]
Finally,  $C\subset\br\setminus \bq$, since $F'\cap G'\subset\br\setminus \bq$.
\end{proof}

\begin{theorem}\label{dar}
 If $f:\br\to\br$ is non-negative and
\[
\sum_{k\in\bz}f(\xi+k)= \sum_{j\in\bz}f(2^j\xi)=1 
\qquad\text{for a.e. }\xi\in\br,
\]
then there exists a set $G\subset\supp(f)$, such that 
\begin{equation} \label{t2}
\sum_{k\in\bz}\ch_G(\xi+k)= \sum_{j\in\bz}\ch_G(2^j\xi)=1 
\qquad\text{for a.e. }\xi\in\br.
\end{equation}
\end{theorem}
\begin{proof}
Let $C$ be the set provided in  Proposition~\ref{propC}.
Consider an equivalence relation on $C$, given as $\xi_1\sim \xi_2$ $\iff$ there exist $j\in\bz$ and $q\in\bq$ such that 
$\xi_1=2^j\xi_2+q$. Since $2^jC+q= C$ for all $j\in\bz$, $q\in\bq$, the equivalence class of $\xi\in C$ is $[\xi]=\{2^j\xi+q: j\in\bz,q\in\bq\}=\{2^j(\xi+q): j\in\bz,q\in\bq\}$.

Let $Z$ be a set that contains exactly one representative from each equivalence class. Fix $\xi\in Z$.
Let $d$ be the dilation projection and let $\tau$ be the translation projection from $\br$ to 
$W=(-1,-\frac12]\cup(\frac12,1]$. That is, for any subset $V \subset \R$, we define
\[
d(V) = \bigcup_{j\in \Z}2^jV \cap W \qquad\text{and}\qquad\tau(V)=\bigcup_{k\in\Z}(V+k)\cap W.
\]

 Since the set $[\xi]$ is countable, the sets $X:=X(\xi)=d([\xi])$ and $Y:=Y(\xi)=\tau([\xi])$ can be written as $X=\{d_m:m\in\bn\}$ and $Y=\{t_n:n\in\bn\}$ with $d_m:=d_m(\xi)$ and $t_n:=t_n(\xi)$. Indeed, if $d(\xi+q)=d(\xi+q')$
for some $q,q'\in\bq$, then $2^j(\xi+q)=\xi+q'$ for some $j\in \bz$. Since $\xi$ is irrational one gets that $j=0$ and $q=q'$. Therefore, the set $X$ is infinite. Similarly, if $\tau(2^j\xi)=\tau(2^{j'}\xi)$ for some $j,j'\in\bz$, then
$2^j\xi+k=2^{j'}\xi$ for some $k\in\bz$. Since $\xi$ is irrational one gets that $j=j'$, and therefore, $Y$ is infinite.  

Assume that for a given pair $(d_m,t_n)\in X\times Y$ there is $j(m,n)\in\bz$ and $q(m,n)\in\bq$ such that
\[
d\big(2^{j(m,n)}(\xi+q(m,n))\big)=d_m \quad\text{ and }\quad \tau\big(2^{j(m,n)}(\xi+q(m,n))\big)=t_n.
\] 

Note that if $d(2^{j}(\xi+q))=d(2^{j'}(\xi+q'))$ for some $j,j'\in\bz$ and $q,q'\in\bq$, then
$2^{j}(\xi+q)=2^{j'+l}(\xi+q')$ for some $l\in\bz$. Thus, $j=j'+l$ since $\xi$ is irrational, and therefore, $q=q'$. This implies, that $q(m,n)$ given above does not depend on $n$, so $q(m,n)=q_m$.

Similarly,  if $\tau(2^{j}(\xi+q))=\tau(2^{j'}(\xi+q'))$ for some $j,j'\in\bz$ and $q,q'\in\bq$, then
$2^{j}(\xi+q)=2^{j'}(\xi+q')+k$ for some $k\in\bz$. Thus, $j=j'$ since $\xi$ is irrational. This implies, that $j(m,n)$ given above does not depend on $m$, so $j(m,n)=j_n$.

In particular, for each pair $(m,n)\in\bn\times\bn$ there is at most one choice of $j_n\in\bz$ and $q_m\in\bq$, such that
\begin{equation}\label{matr}
d\big(2^{j_n}(\xi+q_m)\big)=d_m \quad\text{ and }\quad \tau\big(2^{j_n}(\xi+q_m)\big)=t_n.
\end{equation}

For each $n\in\bn$ one has that $t_n\in Y=\tau([\xi])$. Clearly, $[\xi]=\{2^j(\xi+q): j\in\bz,q\in\bq\}$. Thus, $t_n=\tau(2^j(\xi+q))$ for some $j\in\bz$ and $q\in\bq$. By \eqref{matr} this means that
$j=j_n$ and $q\in\{q_m:m\in\bn\}$ with the values of $q$ depending on the value of $d(2^j(\xi+q))\in X=\{d_m:m\in\bn\}$. 
This leads to a definition of a non-empty set
\[
M(n)=\{ m\in\bn: d\big(2^{j_n}(\xi+q_{m})\big)=d_m \text{ and }  \tau\big(2^{j_n}(\xi+q_m)\big)=t_n\}.
\]

Similarly, for each $m\in\bn$ one has that $d_m\in X= d([\xi])$. Thus, $d_m=d(2^j(\xi+q))$ for some $j\in\bz$ and $q\in\bq$. By \eqref{matr} this means that
$q=q_m$ and $j\in\{j_n:n\in\bn\}$ with the values of $j$ depending on the value of $\tau(2^j(\xi+q))\in Y=\{t_n:n\in\bn\}$. 
This leads to a definition of a non-empty set
\[
N(m)=\{ n\in\bn: \tau\big(2^{j_n}(\xi+q_{m})\big)=t_n \text{ and } d\big(2^{j_n}(\xi+q_{m})\big)=d_m \}.
\]

Moreover, $m\in M(n)$ iff \eqref{matr} holds, which is equivalent to $n\in N(m)$.

Define a matrix $A=(a_{mn})$ by
\[
a_{mn}=
\begin{cases}
f\big(2^{j_n}(\xi+q_m)\big)&\text{if \eqref{matr} holds,}\\
0& \text{otherwise.}
\end{cases}
\]
Note that the definitions provided for the sets $M(n)$ and $N(m)$ give that
\[
a_{mn}=
\begin{cases}
f\big(2^{j_n}(\xi+q_m)\big)& m\in M(n),\\
0& \text{otherwise,}
\end{cases}=\begin{cases}
f\big(2^{j_n}(\xi+q_m)\big)& n\in N(m),\\
0& \text{otherwise}.
\end{cases}
\]

By construction, $A$ is a doubly stochastic matrix. Indeed, for each $n\in\bn$ there is an $m_0\in M(n)$, thus for every $k\in\bz$ one has that 
$\tau\big(2^{j_n}(\xi+q_{m_0})+k\big)=t_n$ and $d\big(2^{j_n}(\xi+q_{m_0})+k\big)\in X$. On the other hand,
if $m_1\in M(n)$, then $\tau\big(2^{j_n}(\xi+q_{m_1})\big)=t_n$, so $2^{j_n}(\xi+q_{m_1})=2^{j_n}(\xi+q_{m_0})+k$ for some $k\in\bz$. Therefore, 
\[
\sum_{m=1}^\infty a_{mn}=\sum_{m\in M(n)}f\big(2^{j_n}(\xi+q_m)\big)=\sum_{k\in\bz}
f\big(2^{j_n}(\xi+q_{m_0})+k\big)=1,
\]
by Proposition~\ref{propC}, since $\xi\in C$, so $2^{j_n}(\xi+q_{m_0})\in C$.

Similarly, for each $m\in\bn$ there exists an $n_0\in N(m)$, thus for every $j\in\bz$ one has that 
$d\big(2^{j_{n_0}+j}(\xi+q_{m})\big)=d_m$ and $\tau\big(2^{j_{n_0}+j}(\xi+q_{m})\big)\in Y$. On the other hand,
if $n_1\in N(m)$, then $d\big(2^{j_{n_1}}(\xi+q_{m})\big)=d_m$, so $2^{j_{n_1}}(\xi+q_{m})=2^{j_{n_0}+j}(\xi+q_{m})$ for some $j\in\bz$. Therefore, 
\[
\sum_{n=1}^\infty a_{mn}=\sum_{n\in N(m)}f\big(2^{j_n}(\xi+q_m)\big)=\sum_{j\in\bz}
f\big(2^{j_{n_0}+j}(\xi+q_{m})\big)=1,
\]
by Proposition~\ref{propC}, since $\xi\in C$, so $2^{j_{n_0}}(\xi+q_{m})\in C$.

Since $A$ is doubly stochastic, by Theorem~\ref{birkhoff} there is a bijection $\sigma:\bn\to\bn$ such that $a_{m\sigma(m)}>0$ for every $m\in\bn$, and therefore, $f\big(2^{j_{\sigma(m)}}(\xi+q_{m})\big)>0$.

Let 
\[
R_\xi=\{2^{j_{\sigma(m)}}(\xi+q_{m}):m\in\bn\} \quad \text{and}\quad G=\bigcup_{\xi\in Z}R_\xi.
\]
Clearly, $G\subset\supp(f)$ and the following four steps are needed to establish that \eqref{t2} holds.

If $2^jR_{\xi_1}\cap R_{\xi_2}\ne\emptyset$ for some $\xi_1,\xi_2\in Z$ and $j\in\bz$, then
$2^{j_{\sigma(m_1)}+j}(\xi_1+q_{m_1})=2^{j_{\sigma(m_2)}}(\xi_2+q_{m_2})$ for some
$m_1,m_2\in\bn$. Thus, $\xi_2\in[\xi_1]$, so $\xi_2=\xi_1$ since $\xi_1,\xi_2\in Z$. The irrationality of $\xi_1$ implies that $j_{\sigma(m_1)}+j=j_{\sigma(m_2)}$, so $q_{m_1}=q_{m_2}$.  By \eqref{matr}, this leads to $d_{m_1}=d_{m_2}$, so $m_1=m_2$. Thus, $\sigma(m_1)=\sigma(m_2)$, which yields $j=0$. Therefore, $2^jG\cap G=\emptyset$ for $j\in\bz\setminus\{0\}$.

Similarly, if $(R_{\xi_1}+k)\cap R_{\xi_2}\ne\emptyset$ for some $\xi_1,\xi_2\in Z$ and $k\in\bz$, then
$2^{j_{\sigma(m_1)}}(\xi_1+q_{m_1})+k=2^{j_{\sigma(m_2)}}(\xi_2+q_{m_2})$ for some
$m_1,m_2\in\bn$. Thus, $\xi_2\in[\xi_1]$, so $\xi_2=\xi_1$. The irrationality of $\xi_1$ implies that $j_{\sigma(m_1)}=j_{\sigma(m_2)}$.  By \eqref{matr}, this leads to $t_{\sigma(m_1)}=t_{\sigma(m_2)}$, so
$\sigma(m_1)=\sigma(m_2)$. This implies that $m_1=m_2$, which yields $k=0$. Therefore, $(G+k)\cap G=\emptyset$ for $k\in\bz\setminus\{0\}$.

To see that $\bigcup_{j\in\bz}2^jG=\br$ modulo null sets, it suffices to show that $d(G)=W$ modulo null sets. Toward this end note that $d(R_\xi)=X(\xi)=d([\xi])$ for $\xi\in Z$, since $d\big(2^{j_{\sigma(m)}}(\xi+q_{m})\big)=d_m(\xi)$    by \eqref{matr}. Thus,
\[
d(G)=\bigcup_{\xi\in Z}d(R_\xi)=\bigcup_{\xi\in Z}d([\xi])=d\bigg(\bigcup_{\xi\in Z}[\xi]\bigg)=d(C).
\]
Since the set $C$ is of full measure, one gets that $d(G)=d(C)=W$ modulo null sets.

Similarly, to see that $\bigcup_{k\in\bz}(G+k)=\br$ modulo null sets, it suffices to show that $\tau(G)=W$ modulo null sets. Toward this end note that $\tau(R_\xi)=Y(\xi)=\tau([\xi])$ for $\xi\in Z$, since $\tau\big(2^{j_{\sigma(m)}}(\xi+q_{m})\big)=t_{\sigma(m)}(\xi)$    by \eqref{matr}, and $\sigma$ is a permutation of $\bn$. Thus,
\[
\tau(G)=\bigcup_{\xi\in Z}\tau(R_\xi)=\bigcup_{\xi\in Z}\tau([\xi])=\tau\bigg(\bigcup_{\xi\in Z}[\xi]\bigg)=\tau(C).
\]
Since the set $C$ is of full measure, one gets that $\tau(G)=\tau(C)=W$ modulo null sets.

The four steps provided above lead to the conclusion that \eqref{t2} holds for a.e. $\xi\in\br$. 
\end{proof}

\begin{remark}
In Theorem~\ref{dar} measurability of $f$ and $G$ is not addressed, $f$ is defined in each point of $\br$, $\supp(f)=\{\xi\in\br:f(\xi)\ne0\}$ and $G$ is possibly non-measurable.
\end{remark}

\begin{example}
Let $F \subset [1/3, 2/3]$ be a non-measurable set. Define $G = [1/3, 2/3] \setminus F$. Then, $F \cup 2G \cup (G - 1) \cup 2(F - 1)$ is a non-measurable set which tiles by translations and dilations.
\end{example}

Note that if a nonnegative, measurable function $f$ tiles by translations and dilations, then $\int f \, dm = 1$ and $\int f\, d\nu = 1$, where $m$ is the Lebesgue measure and the measure $\nu$ is defined on the Lebesgue measurable sets $V$ by $\nu(V) = \int_V \frac{dx}{|x| \log 4}$. We showed in Example \ref{ex:toostrong} that the two integral conditions do not imply that the support of $f$ contains a wavelet set. However, if one were trying to find a negative solution to Problem \ref{mainprob} or Problem \ref{stronglarson} in light of Theorem \ref{dar}, one might hope to find an example of a function which tiles by both translations and dilations whose support does not contain sets of the correct measures. That is, we are interested in knowing whether every nonnegative, measurable $f$ that tiles by translations and dilations has support that contains a set $F$ with $m(F) = 1$ and $\nu(F) = 1$. Theorem \ref{marcin_meas} provides the expected positive answer.

\begin{lemma}\label{lemma:ge}
If $U$ and $V$ are measurable subsets of $\R$ such that $m(U) = m(V) = c_1$ and $\nu(U) \le c_2 \le \nu(V)$, then there exists a set $W \subset U \cup V$ such that $m(W) = c_1$ and $\nu(W) = c_2$, where $0 \le c_i < \infty$. 
\end{lemma}

\begin{proof}
Without loss of generality, we can assume that $\nu(V) < \infty$. In fact, if $\nu(V) = \infty$, then for every $\epsilon > 0$, $\nu(V \cap [-\epsilon, \epsilon]) = \infty$. Choose $\epsilon > 0$ such that $\nu(V \setminus [-\epsilon, \epsilon]) > 1$ and let $V_1 = V \setminus [-\epsilon, \epsilon].$ Note that
\[
m(U \setminus V_1) \ge m(V) -m(V_1) = m(V \cap [-\epsilon,\epsilon]).
\]
Hence, there exists a subset $V_2$ of $U \setminus V_1$ of measure $m(V \cap [-\epsilon, \epsilon])$. Once the lemma is proven for $U$ and $V_1 \cup V_2$, then it also holds for the original $U$ and $V$.

Next, we show that it suffices to prove the lemma when $U$ and $V$ are disjoint. Let $G = U \cap V$, $U^\prime = U \setminus G$, and $V^\prime = V \setminus G$. Note that \[
m(U')=m(V')=c_1-m(G)\quad\text{and}\quad \nu(U')\le c_2 - \nu(G) \le \nu(V').
\]
If we can find $W^\prime \subset U^\prime \cup V^\prime$ such that $m(W^\prime) = m(U^\prime) = m(V^\prime)$ and $\nu(W^\prime) = c_2 - \nu(G)$, then $W = W^\prime \cup G$ satisfies $m(W) = c_1$ and $\nu(W) = c_2$.

Assume $U$ and $V$ are disjoint. Consider $U_t = U \cap [-t, t]$. Let $t_0 = \sup\{t \ge 0: m(U_t) = 0\}$. For each $t > t_0$, there exists $g(t) > 0$ such that $m(V\cap [-g(t), g(t)]) = m(U_t)$. Define $V_t = V\setminus [-g(t), g(t)]$ and let $W_t = U_t \cup V_t$. We have $m(W_t) = 1$ for all $t$, and $\nu(W_t) \to \nu(V)$ as $t\to t_0$ and $\nu(W_t) \to \nu(U)$ as $t \to \infty$.  Since the function $t\mapsto \nu(W_t)$ is continuous, the lemma follows from the Intermediate Value Theorem.
\end{proof}

\begin{theorem}\label{marcin_meas}
    Let $f$ be a measurable function on $\R$ such that $0 \le f \le 1$ almost everywhere and
 \[
\int_\R f\, dm = c_1 \qquad\text{and}\qquad
\int_\R f\, d\nu = c_2,
\]
    where $0 \le c_i < \infty$.
    
    Then, there exists a measurable set $U \subset E = \supp(f)$ such that $m(U) = c_1$ and $\nu(U) = c_2$.
\end{theorem}

\begin{proof}
    We start with some reductions. 
    If $c_1$ or $c_2$ are 0, then the result is obvious. Without loss of generality, the support of $f$ can be assumed to be contained in $[0, \infty)$. Let $F = \{x: f(x) = 1\}$ and note that $\int f = \int_F f + \int_{F^c} f$, so by replacing $c_1$ with $c_1 - \int_F f\, dm$ and $c_2$ with $c_2 - \int_F f \, d\nu$, we can assume that $0 \le f < 1$ almost everywhere, and $f$ is positive on a set of positive measure $E \subset [0, \infty)$.

    Choose $x_0$ such that 
    \[
      m\left(E \cap [0, x_0]\right) = \int_0^{x_0} \ch_E \, dm = c_1  
    \]
    and see that this implies
    \[
    \int_0^{x_0} \left(\ch_E - f\right) \, dm= \int_{x_0}^\infty f\, dm.
    \]
    Let $g(x) = \frac{1}{x \log 4}$ and see that
    \begin{align*}
        \int_0^{x_0} \left(\ch_E - f\right) \, d\nu &\ge \int_0^{x_0} \left(\ch_E - f\right) g(x_0) \, dm\\
        &= \int_{x_0}^\infty f g(x_0) \, dm \\
        &\ge \int_{x_0}^\infty f \, d\nu.
    \end{align*}
    Therefore, 
    \[
     \nu\left(E \cap [0, x_0]\right) \ge \int_0^\infty f\, d\nu = c_2.
    \]
    Therefore, $E \cap [0, x_0]$ satisfies the role of $V$ in Lemma \ref{lemma:ge}.

    Next, suppose that $m(E)<\infty$. Then, there exists $x_1$ such that 
    \[
      m\left(E \cap [x_1, \infty)\right) = \int_{x_1}^{\infty} \ch_E \, dm = c_1.
    \]
    As before, we see that 
        \[
    \int_{x_1}^{\infty} \left(\ch_E - f\right) \, dm= \int_{0}^{x_1} f\, dm.
    \]
    Note that
    \begin{align*}
        \int_{x_1}^{\infty} \left(\ch_E - f\right) \, d\nu &\le \int_{x_1}^{\infty} \left(\ch_E - f\right) g(x_1) \, dm\\
        &= \int_{0}^{x_1} f g(x_0) \, dm \\
        &\le \int_{0}^{x_1} f \, d\nu.
    \end{align*}
    Hence, 
    \[
     \nu\left(E \cap [x_1, \infty)\right) \le \int_0^\infty f\, d\nu = c_2.
    \]
    Therefore, $E \cap [x_1, \infty)$ satisfies the role of $U$ in Lemma \ref{lemma:ge}.

    Finally, if $m(E) = \infty$, then for $x \ge 0$, let $E_x$ be a set of Lebesgue measure $c_1$ such that $E_x \subset [x, \infty)$. Note that $\nu(E_x) < m(E_x) g(x) \to 0$ as $x\to \infty$, so there is an $x$ for which $\nu(E_x) \le c_2$. This $E_x$ can play the role of $U$ in Lemma \ref{lemma:ge}.

    Therefore, by Lemma \ref{lemma:ge}, the support of $f$ contains a set of Lebesgue measure $c_1$ and $\nu$ measure $c_2$, as desired.
\end{proof}

\bibliographystyle{plain}
\bibliography{bibliography.bib}

\end{document}